\documentclass[a4paper,reqno]{amsart}

\usepackage[usenames,dvipsnames,svgnames,table]{xcolor}

\usepackage{hyperref}
\usepackage{amsrefs}

\usepackage[framemethod=tikz]{mdframed}

\usepackage{tikz}
\usepackage{tikz-cd}

\usepackage[active]{srcltx}
\usepackage{amssymb}
\usepackage{mathrsfs}

\newtheorem{teo}{Theorem}[section]
\newtheorem{defin}[teo]{Definition}
\newtheorem{prop}[teo]{Proposition}

\theoremstyle{definition}
\newtheorem{remark}[teo]{Remark}

\newtheorem{say}[teo]{}

\numberwithin{equation}{section}

\newcommand{\C}{\mathbb{C}}
\newcommand{\R}{\mathbb{R}}
\newcommand{\Zeta}{{\mathbb{Z}}}
\newcommand{\meno}{^{-1}}
\newcommand{\alfa}{\alpha}
\newcommand{\alf}{\alpha}
\newcommand{\la}{\lambda}
\newcommand{\lam}{\lambda}
\newcommand{\restr}[1] {\vert_{#1}}
\newcommand{\cinf}{C^\infty}
\newcommand{\om}{\omega}
\newcommand{\lds}{\ldots}
\newcommand{\cds}{\cdots}
\newcommand{\cd}{\cdot}
\newcommand{\sx}{\langle}
\newcommand{\xs}{\rangle}
\newcommand{\scalo}{\sx \ , \ \xs}
\newcommand{\lra}{\longrightarrow}
\newcommand{\ga}{\gamma}
\newcommand{\Ga}{\Gamma}
\newcommand{\harm}{\mathscr{H}}
\newcommand{\est} {\Lambda}
\newcommand{\forme} {\mathcal A}

\newcommand{\sieg}{\mathfrak{H}}
\newcommand{\univab}{\mathcal{A}}
\newcommand{\univabt}{{\mathfrak{A}}}
\newcommand{\varpit}{\varpi}
\newcommand{\Abel}{\alfa}
\newcommand{\Abelz}{\Abel_{x_0}}
\newcommand{\siegg}{\mathfrak{H}_g}

\newcommand{\LB}{\Theta}
\newcommand{\LBD}{\LB^\otwo}
\newcommand{\M}{\mathsf{M}}
\newcommand{\Mg}{{\mathsf{M}_g}}
\newcommand{\A}{\mathsf{A}}

\newcommand{\Z}{{\mathbb{Z}}}
\newcommand{\PGL}{\operatorname {PGL}}
\newcommand{\Sp}{\operatorname {Sp}}
\newcommand{\barz}{\overline{z}}
\newcommand{\debar }{\overline{\partial } }
\newcommand{\de }{ {\partial } }

\def\delbar{\overline{\partial }}

\newcommand{\erre}{\operatorname{R}}

\newcommand{\Cg}{{\mathcal{C}_g}}
\newcommand{\Conn}{\operatorname{Conn}}
\newcommand{\ta}{\tau}
\newcommand{\vbE}{E}
\newcommand{\otwo}{{\otimes 2}}

\newcommand{\vol}{\operatorname{vol}}
\newcommand{\prost}{{\mathcal P_g}}
\newcommand{\Curv}{\erre^{1,1} } 
\newcommand{\TT}{\operatorname{T}}
\newcommand{\mut}{{\frac{\sqrt{-1}}{12\pi}}}
\newcommand{\muto}{{\frac{\sqrt{-1}}{2\pi}}}
\newcommand{\class}{{\Lambda}}
\newcommand{\WP}{{\operatorname{WP}}}

\usepackage{version}

\includeversion{longversion}

\begin{document}
	
\title[Theta bundle and the Hodge theoretic projective structure]{Theta bundle,
Quillen connection and the Hodge theoretic projective structure}
	
\author[I. Biswas]{Indranil Biswas}

\address{Department of Mathematics, Shiv Nadar University, NH91, Tehsil
Dadri, Greater Noida, Uttar Pradesh 201314, India}

\email{indranil.biswas@snu.edu.in, indranil29@gmail.com}

\author[A. Ghigi]{Alessandro Ghigi}
	
\address{Dipartimento di Matematica, Universit\`a di Pavia, via
Ferrata 5, I-27100 Pavia, Italy}
	
\email{alessandro.ghigi@unipv.it}
	
\author[C. Tamborini]{Carolina Tamborini}

\address{Essener Seminar f\"ur Algebraische Geometrie und Arithmetik, Fakult\"at f\"ur Mathematik, 
Universit\"at Duisburg–Essen, 45117 Essen, Germany}

\email{carolina.tamborini@uni-due.de}


	
\begin{abstract}
There are two canonical projective structures on any compact Riemann surface of genus at least two: one 
coming from the uniformization theorem, and the other from Hodge theory. They produce two (different) families 
of projective structures over the moduli space $\M_g$ of compact Riemann surfaces. A recent work of Biswas, 
Favale, Pirola, and Torelli shows that families of projective structures over $\M_g$ admit an 
equivalent characterization in terms of complex connections on the dual $\mathcal{L}$ of
the determinant of the Hodge line bundle 
over $\M_g$; the same work gave the connection on $\mathcal L$ corresponding to the projective
structures coming from uniformization.\\ Here we construct the 
connection on $\mathcal L$ corresponding to the family of Hodge theoretic projective structures. This
connection is described in three different
ways: Firstly as the connection induced on $\mathcal L$ by the Chern connection of the $L^2$-metric 
on the Hodge bundle, secondly as an appropriate root of the Quillen metric induced by the (square of the) 
Theta line bundle on the universal family of abelian varieties, endowed with the natural Hermitian metric 
given by the polarization, and finally as Quillen connection gotten using the Arakelov metric on the 
universal curve, modified by Faltings' delta invariant.
\end{abstract}
\maketitle

\tableofcontents
	
\section{Introduction}

A projective structure on a Riemann surface $C$ is an equivalence class
of projective atlases, i.e., an equivalence class of coverings by
holomorphic coordinate charts such that all the transition functions
are M\"obius transformations. Equivalently, a projective structure on $C$
is a holomorphic principal $\text{PGL}(2, {\mathbb C})$--bundle $E_{\text{PGL}}$ on $C$
equipped with a holomorphic connection $D$ and also a holomorphic reduction of structure group
$E_B\, \subset\, E_{\text{PGL}}$ to the subgroup of lower triangular matrices (the Borel
subgroup of $\text{PGL}(2, {\mathbb C})$) such that the second fundamental form of it, for the
connection $D$, is everywhere nonzero. Projective structures on $C$ form an affine space
$\mathcal P(C)$ modelled on $H^0(C,\, 2K_C)$.

We consider compact connected Riemann surfaces.
If $C$ varies in the moduli space $\Mg$ of compact
Riemann surfaces of genus $g$ with $g\, \geq\, 2$ (which is
considered as a complex analytic orbifold), these spaces $\mathcal P(C)$
of projective structures produce a holomorphic fiber bundle
\begin{gather*}
\phi \,\,:\,\, \mathcal P_g \,\,\lra\,\, \Mg,
\end{gather*}
which is in fact a holomorphic torsor for the holomorphic cotangent
bundle $\Omega^1_{\Mg}$; see \cite{BFPT}. There is a natural
$C^\infty$ section $\beta^u$ of $\mathcal P_g$, which associates to
each $[C]\,\in\, \Mg$ the projective structure obtained from the
uniformization theorem. This classical projective structure is
called \emph{canonical} because it depends only on $C$, and not on
any extra choice or data. On the other hand, there is another such
canonical global $C^\infty$ section $\beta^\eta$ of $\mathcal P_g$,
different from $\beta^u$, which is constructed using only basic
Hodge theory \cite{BCFP}; this construction, which is based on
\cite{cpt} and \cite{cfg}, is briefly recalled in Section 4.1.

Let
\begin{gather*}
\mathcal L \,\,\lra\,\, \Mg
\end{gather*}
be the holomorphic line bundle obtained by taking the dual of the determinant of the Hodge
bundle over $\M_g$. In \cite{BFPT} it is proved that the families of projective
structures over $\M_g$ admit an equivalent characterization in terms
of complex connections on $\mathcal L$. The idea is the following. First note that
connections on a vector bundle form an affine space. The space of
connections on $\mathcal L$ whose $(0,\, 1)$-part is compatible with
the holomorphic structure of $\mathcal L$ coincides with the space
of $C^\infty$ sections of a holomorphic fiber bundle
\begin{gather*}
\Conn (\mathcal L) \,\,\lra\, \,\Mg
\end{gather*}
which is a holomorphic $\Omega^1_\Mg$-torsor. If follows from a
recent result of \cite{FPT} on holomorphic 1-forms on $\Mg$ (see
Theorem \ref{teo-fpt} below) that for $g\,\geq\, 5$ there is at most
one holomorphic isomorphism between the above $\Omega^1_\Mg$-torsors
$\Conn(\mathcal L)$ and $\mathcal P_g$. In \cite{BFPT} it is proved
that such a holomorphic isomorphism does exist. In this way one gets
a canonical correspondence between global sections of $\mathcal P_g$
and connections on $\mathcal L$ whose whose $(0,\, 1)$-part is
compatible with the holomorphic structure of $\mathcal L$. Moreover,
for a global $C^\infty$ section $\beta$ of $\mathcal P_g$ one can
define a $(1,\,1)$-form $\debar \beta$. Hence we get a diagram
\begin{equation}\label{intro-diag}
\begin{tikzcd}
\Gamma (\Conn (\mathcal L)) \arrow[rr] \arrow[dr]
& & \Gamma (\mathcal P_g) \arrow[dl] \\
& \forme^{1,1}(\Mg). &
\end{tikzcd}
\end{equation}
The horizontal arrow comes from the isomorphism of torsors, and it
is a bijection. The left arrow assigns to a connection its
curvature, and the right arrow is given by the above mentioned $\debar$
operation multiplied by $\frac{\sqrt{-1}}{2\pi}$; these two arrows are injective. In Section
\ref{sec:correspondence}, in particular in Theorem
\ref{thm:correspondence}, some observations on this correspondence
are made.

Let $\nabla^Q$ be the Quillen connection on $\mathcal L$ constructed
using the Poincar\'e metric on the fibers of the universal curve
over $\Mg$. It was shown in \cite{BFPT} that the horizontal
bijection in \eqref{intro-diag} takes $\nabla^Q$ to the section
$\beta^u$, and the corresponding curvature in the diagram
\eqref{intro-diag}, which coincides with ${\frac{\sqrt{-1}}{2\pi}}\debar \beta^u$,
is actually the Weil-Petersson K\"ahler form on
$\Mg$. This completely describes the position of $\beta^u$ in the triangle
\eqref{intro-diag}, and connects three very natural objects which
are related to the uniformization and the Poincar\'e metric.

Here we study the earlier mentioned section $\beta^{\eta}$ of Hodge
theoretic origin. It was proved in \cite{BCFP} that
$\debar \beta^\eta$ is a multiple of $j^*\om_S$, where
$j\,:\, \Mg\,\lra\, \A_g$ is the period map and $\om_S$ is the
Siegel metric on $\A_g$. We look for the connection on $\mathcal L$
that corresponds to $\beta^\eta \,\in\, \Gamma (\mathcal P_g)$ (in
the diagram \eqref{intro-diag}) and
$j^*\om_S \,\in\, \forme^{1,1}(\Mg)$. Since the projective structure
$\beta^{\eta}([C])$ is constructed using basic Hodge theory on $C$,
it is natural to expect its corresponding connection to be related
to the pull-back, via the period map, of some natural hermitian
line bundle on the moduli space $\A_g$. This is indeed the case. We show
that this connection can be described in at least three different
ways.

In Section \ref{secl2} we describe the above connection associated to
$\beta^{\eta}$ as the one
induced on $\mathcal L$ by the Chern connection of the $L^2$-metric
on the Hodge bundle. In Section \ref{sec:abelian-varieties} we
obtain the connection as the Chern connection of an appropriate root
of the Quillen metric induced by the (square of the) Theta line
bundle on the universal family of abelian varieties, endowed with
the natural Hermitian metric given by the polarization.

Finally, we address the question 
whether it is possible to obtain the same connection on
$\mathcal{L}$ as a Quillen connection constructed using some metric
on the fibers of the universal curve, as it happens with the
Poincar\'e metric in the case of the projective structure coming
from uniformization. In this direction, we observe in Section
\ref{sec:curv-arak-metr} that one can obtain the connection on
$\mathcal L$ corresponding to $\beta^{\eta}$ as a modification of
the Quillen connection gotten on $\mathcal L$ using the Arakelov
metric on the universal curve; the modification is done using
Faltings' delta invariant.

\section{Projective structures and connections}\begin{say}\label{cptriemann}
		
Let $C$ be a compact connected Riemann surface of genus $g$, with
$g\,\geq\, 2$. Denote by $K_C$ the canonical line bundle of $C$. A
holomorphic atlas on $C$ is \emph{projective} if for every pair of
overlapping charts the change of coordinates on each connected
component is the restriction of some M\"obius transformation (i.e.,
element of $\PGL(2,\C)$). Two projective atlases are
\emph{equivalent} if their union is also a projective atlas. A
\emph{projective structure} is an equivalence class of projective atlases.

Another description of projective structures: A projective structure on $C$ is
a holomorphic ${\mathbb C}{\mathbb P}^1$--bundle ${\mathbb P}\, \longrightarrow\, C$ together
with a holomorphic connection $D$ on $\mathbb P$, and a holomorphic section $\sigma\,:\, C\,
\longrightarrow\, {\mathbb P}$, such that the second fundamental form of $\sigma$, for the
connection $D$, is everywhere nonzero. Consider the trivial ${\mathbb C}{\mathbb P}^1$--bundle
${\mathbb C}{\mathbb P}^1\times{\mathbb C}{\mathbb P}^1 \, \longrightarrow\, {\mathbb C}{\mathbb P}^1$
equipped with the trivial connection and the section given by the diagonal ${\mathbb C}{\mathbb P}^1
\, \subset\, {\mathbb C}{\mathbb P}^1\times{\mathbb C}{\mathbb P}^1$.
Given a projective structure on $C$ defined using coordinate charts as above, pullback this flat
${\mathbb C}{\mathbb P}^1$--bundle with section to the domains of the coordinate charts. Using the
transition functions for the charts, these locally defined flat ${\mathbb C}{\mathbb P}^1$--bundles with section
patch together compatibly to define a holomorphic ${\mathbb C}{\mathbb P}^1$--bundle on $C$ equipped
with a holomorphic connection and a holomorphic section. (See \cite{Gu}.)

Let $\mathcal{P}(C)$ denote the space of all projective structures
on $C$. Then $\mathcal{P}(C)$ is an affine space modelled on
$H^0(C,\, 2K_C)$ \cite{Gu}. There are several equivalent
characterizations of projective structures; here we will use the
following one from \cite{BR1}.

Consider the complex surface $C\times C$, and let
$\Delta \,\subset\, {C\times C}$ be the reduced diagonal on it. For
$k\,\geq\, 0$, denote by $\Delta_{k+1}$ the $k$--th order
infinitesimal neighborhood of $\Delta$ in ${C\times C}$.  Consider
the line bundle
\begin{equation}\label{dl}
K_{C\times C}(2\Delta)\,:=\, K_{C\times C}\otimes {\mathcal O}_{C\times C}(2\Delta)
\,\lra\, {C\times C}.
\end{equation}
By \cite[eq. (3.3) p. 758]{BR1} there is a short exact sequence
\begin{gather*}
0\,\longrightarrow\, K_C^{\otwo} \,\longrightarrow\, K_{C\times
C}(2\Delta)\big\vert_{\Delta_3}
\,\stackrel{\Psi'}{\longrightarrow}\, K_{C\times
C}(2\Delta)\big\vert_{\Delta_2}\,\longrightarrow\, 0,
\end{gather*}
where $\Psi'$ is the restriction map. Since
$H^1(C,\,K_C^{\otwo})\,=\, 0$ by the assumption that $g\,\geq\, 2$,
this gives the exact sequence of cohomologies
\begin{gather}\label{h}
0\,\longrightarrow\, H^0(C,\, 2K_C) \,\longrightarrow\,
H^0\left({\Delta_3},\, K_{C\times
C}(2\Delta)\big\vert_{\Delta_3}\right)
\,\overset{\widetilde\Psi}{\longrightarrow}\, H^0\left(\Delta_2,\,
K_{C\times C}(2\Delta)\big\vert_{\Delta_2}\right)
\,\longrightarrow\, 0,
\end{gather}
where $\widetilde\Psi$ is induced by $\Psi'$. By \cite[Theorem
2.1]{BR1}, \cite[Theorem 2.2]{BR2}, there is a canonical
trivializing section of $K_{C\times C}(2\Delta)\big\vert_{\Delta_2}$
\begin{equation}\label{s0}
s_{0, C}\,\,\in\,\, H^0\left(\Delta_2,\, K_{C\times C}(2\Delta)\big\vert_{\Delta_2}\right),
\end{equation}
and by \cite[Theorem 3.2]{BR1}, there is a canonical bijection
\begin{equation}\label{phi}
\mathcal{P}(C)\, \longrightarrow\, \,\,\widetilde{\Psi}^{-1}(s_{0,C}),
\end{equation}
where $\widetilde\Psi$ is the map in \eqref{h}. We will identify
$\mathcal{P}(C)$ and $\widetilde{\Psi}^{-1}(s_{0,C})$ using this
bijection. From \eqref{h} it follows that $\mathcal{P}(C)$ is an
affine space for $H^0(C, \,2K_C)$. See also \cite[Theorem 3.2]{BR1}
and \cite[p.~688, Theorem 2.2]{BR2}.
\end{say}

\begin{say}\label{families}
Following \cite[\S~2]{BCFP} one can globalize the above construction.

Let $\Mg$ denote the
moduli space of compact Riemann surfaces of genus $g$.
First of all consider on $\Mg$ the complex topology and endow it
with the structure of a complex analytic orbifold in the following
way. If $[C]$ is a point of $\Mg$, there is a Kuranishi family
$\pi \,:\,\mathcal C \lra B$, where $B$ is a polydisk, $0 \,\in\, B$ and
$C_0 \,\cong\, C$. The group of holomorphic automorphisms $\operatorname{Aut}(C_0)$ acts on both
$\mathcal C$ and $B$; the map $\pi$ is equivariant for these actions of $\operatorname{Aut}(C_0)$.
One can show that the spaces $(B, \,\operatorname{Aut}(C_0))$ are the uniformizers of a complex analytic
orbifold structure on $\Mg$. (See \cite[XI.4 and XII.4]{acg2} for more details.) As it is well known,
there is no universal family on $\Mg$. But most local geometric constructions
can be performed on the Kuranishi family $\mathcal C \,\lra\, B$ and
this way we get appropriate orbifold objects on $\Mg$. So there is a universal family
\begin{gather}\label{pi}
\pi\,\,:\,\, \Cg \,\, \lra\,\, \Mg
\end{gather}
in the orbifold category. For example, as shown in \cite[\S~2]{BCFP},
the spaces
$ H^0({\Delta_3},\, K_{C_t\times C_t}(2\Delta)\vert_{\Delta_3}) $ for
$t\,\in \,B$ glue together and form a holomorphic
$\operatorname{Aut}(C_0)$-equivariant vector bundle over $B$. Moreover
these local bundles give rise to a holomorphic orbifold vector bundle
$ \mathcal V$ on $\Mg$. Similarly the spaces
$H^0(\Delta_2,\, K_{C_t\times C_t}(2\Delta)\vert_{\Delta_2}) $ give
rise to a holomorphic orbifold vector bundle $\mathcal V_2$ on $\Mg$.
Summing up from \eqref{h} we get an exact sequence of holomorphic
orbifold vector bundles on $\Mg$
\begin{gather}\label{ecv}
0\,\longrightarrow\, \Omega^1_\Mg\,=\,\pi_* K^{\otimes
2}_{\Cg/\Mg} \,\longrightarrow\, \mathcal{V}
\,\overset{\Psi}{\longrightarrow}\, \mathcal{V}_2
\,\longrightarrow\, 0,
\end{gather}
where
$$
\mathcal V_{[C]} \,=\, H^0\left({\Delta_3},\, K_{C\times
C}(2\Delta)\big\vert_{\Delta_3}\right), \ \ \, \mathcal
V_{2,[C]}\, = \,H^0\left({\Delta_2},\, K_{C\times
C}(2\Delta)\big\vert_{\Delta_2}\right),
$$
and $\Psi$ is given by $\widetilde\Psi$ in \eqref{h}. Set
\begin{equation}\label{l1}
\mathcal{P}_g\,:=\, \Psi\meno (\mathbf{s}_0),
\end{equation}
where $\mathbf{s}_0$ is the section of $\mathcal V_2$ defined by
$\mathbf{s}_0 ([C])\,:=\, s_{0,C}$ (see \eqref{s0}). Let
\begin{gather}\label{wv}
\phi \,\,:\,\, {\mathcal P}_g \,\,\lra\,\, \Mg
\end{gather}
be the restriction of the natural projection
$\mathcal V \,\lra\, \Mg$.

Let
\begin{equation}\label{a1}
\widetilde{\mathcal P}_g \,\,\lra\,\, \Mg
\end{equation}
be the moduli space of projective structures; so the elements of $\widetilde{\mathcal P}_g$ are
pairs of the form $({\mathbb X},\, {\mathbb P})$, where ${\mathbb X}\, \in\,\Mg$ and $\mathbb P$
is a projective structure on ${\mathbb X}$. The projection in \eqref{a1} sends any
$({\mathbb X},\, {\mathbb P})$ to $\mathbb X$.

By \eqref{phi}, the inverse image $\phi \meno ([C])$, where $\phi$ is the projection in \eqref{wv}, is in
bijective correspondence with the set of projective structures on
$C$; see Lemma 2.1 in \cite{BCFP} (where ${\mathcal P}_g$ is denoted
by $\widehat{\mathcal V}$). Consequently, we get a holomorphic isomorphism of fiber bundles
\begin{equation}\label{a2}
{\mathcal P}_g \,\,\stackrel{\sim}{\lra}\,\, \widetilde{\mathcal P}_g
\end{equation}
over $\Mg$.

We note that $\phi\,:\,\prost \,\lra\, \Mg$ in \eqref{wv} is a
holomorphic affine fiber bundle and in fact it is a holomorphic
$\Omega^1_\Mg$--torsor over $\Mg$ (see \cite{BFPT}). On the other hand, $\widetilde{\mathcal P}_g$ 
in \eqref{a1} is also a holomorphic $\Omega^1_\Mg$--torsor over $\Mg$; this follows from the fact that
the space of all projective structures on a Riemann surface ${\mathbb X}$ is an affine space for the
space of all holomorphic quadratic differentials on ${\mathbb X}$. The isomorphism in
\eqref{a2} is a holomorphic isomorphism $\Omega^1_\Mg$--torsors with the $\Omega^1_\Mg$--torsor
structure scaled by a factor $6$; see \cite[Lemma 3.6]{BR1}.
	
Let $\beta\,:\, \M_g \,\longrightarrow\, \prost$ be a $C^{\infty}$
section of $\phi$ in \eqref{wv}. So, $\beta$ is a $C^{\infty}$ section
of $\mathcal{V}$ in \eqref{ecv} such that
$\Psi(\beta)\,=\, \mathbf{s}_0$ (see \eqref{ecv} and \eqref{l1}). Since
$\mathbf{s}_0$ is holomorphic, we have
\begin{gather}\label{dv}
\delbar_{\mathcal{V}}(\beta)\,\in\, \forme^{0,1}(\M_g,\,
\Omega^1_{\M_g})\,=\, \forme^{1,1} ({\M_g}).
\end{gather}
In other words, the differential of any $C^{\infty}$ family of
projective structures $\beta\,:\, \M_g \,\lra\, \widehat{\mathcal{V}}$
is a $(1,\,1)$-form on $\M_g$. For notational convenience, the
$(1,\,1)$-form $\delbar_{\mathcal{V}} (\beta)$ in \eqref{dv} will be
denoted by $\delbar \beta$.

The uniformization of Riemann surfaces produces a canonical $\cinf$
section $\beta^u$ of the bundle in \eqref{wv}. Zograf and Takhtadzhyan, \cite{ZT},
proved that the associated $(1,\,1)$-form
$\delbar \beta^u$ has the following expression:
\begin{gather}\label{eq:ZTWP}
\delbar \beta^u\ =\ \frac{1}{6} \, \omega_{\rm WP}, 
\end{gather}
where $\omega_{\rm WP}$ is the Weil-Petersson form on $\M_g$. It should be mentioned
that Zograf and Takhtadzhyan compute the $\debar$-derivative of the section of the space of
projective structures considered as a section of $\widetilde{\mathcal P}_g$ in \eqref{a1}, while
$\beta^u$ denotes the corresponding section of ${\mathcal P}_g$ under the
isomorphism in \eqref{a2}. As mentioned before, both ${\mathcal P}_g$ and $\widetilde{\mathcal P}_g$
are $\Omega^1_\Mg$--torsors over $\Mg$, but the isomorphism in \eqref{a2} is an
isomorphism of $\Omega^1_\Mg$--torsors only when the
$\Omega^1_\Mg$--torsor structure is scaled by a factor $6$.
This is the reason for this factor in \eqref{eq:ZTWP}.
\end{say}
	
\begin{say}\label{se2.3}
\label{connections} Let $B$ be a complex manifold, and let
$E\,\lra\, B$ be a holomorphic vector bundle. We use the standard
notation $\mathcal{A}({E})$, $\mathcal{A}^{k}({E})$,
$\mathcal{A}^{p,q}({E})$ for $C^\infty$ sections of $E$ and
differential forms with values in ${E}$. We say that a connection
$\nabla$ on ${E}$ is \emph{complex} if
$\nabla^{0,1}\,=\, \delbar_{{E}}$, where $\delbar_{{E}}$ is the
Dolbeault operator associated to the holomorphic structure of $E$.

Let $\Phi\, :\, E_{\text{GL}(r)}\, \lra\, B$ be the
holomorphic principal $\text{GL} (r,{\mathbb C})$--bundle
corresponding to $E$, where $r\,=\, \text{rank}(E)$. So the fiber of
$E_{\text{GL}(r)}$ over any $x\, \in\, B$ is the space of all $\mathbb C$--linear
isomorphisms from ${\mathbb C}^r$ to the fiber $E_x$ of $E$ over
$x$. Let $TE_{\text{GL}(r)}$ be the holomorphic tangent bundle of
$E_{\text{GL}(r)}$. The action of $\text{GL}(r,{\mathbb C})$ on
$E_{\text{GL}(r)}$ produces an action of $\text{GL}(r,{\mathbb C})$
on $TE_{\text{GL}(r)}$. The corresponding quotient
$(TE_{\text{GL}(r)})/\text{GL}(r,{\mathbb C})$ is a holomorphic
vector bundle on $B$ (see \cite{At}).  This vector bundle
$(TE_{\text{GL}(r)})/\text{GL}(r,{\mathbb C})$ on $B$ is denoted by
$\operatorname{At}^1(\vbE)$. The differential
$$
d\Phi\, :\, TE_{\text{GL}(r)}\, \longrightarrow\, \Phi^* TB
$$
of $\Phi$, being $\text{GL}(r,{\mathbb C})$--equivariant, produces a
projection
$$
\operatorname{At}^1(\vbE)\,:=\,
(TE_{\text{GL}(r)})/\text{GL}(r,{\mathbb C})\, \longrightarrow\,
(\Phi^* TB)/\text{GL}(r,{\mathbb C})\,=\, TB.
$$
This projection fits in a short exact sequence of holomorphic vector
bundles over $B$
\begin{gather}\label{atiyah}
0\,\lra \,\operatorname{End}(E)\, \lra\, \operatorname{At}^1(\vbE)\,
{\lra}\,TB \,\lra\, 0
\end{gather}
(see \cite{At}); note that
$(\text{kernel}(d\Phi))/\text{GL}(r,{\mathbb C})$ is identified with
$\operatorname{End}(E)$. For an open subset $U\,\subset\, B$, complex
connections on $E\big\vert_{U}$ are in bijection with the smooth
splittings of \eqref{atiyah}. (See especially \cite[Theorem 1]{At} and
\cite[Proposition 9]{At}.) See \cite{EFW} for relationship between families
of connections and other constructions.

Now set $B\,=\, \M_g$. If $E$ is a holomorphic line bundle on $\Mg$, then
${\rm End}(E) =\, \,{\mathcal O}_{\M_g}$. Thus in that case dualizing
\eqref{atiyah} we get
\begin{equation}\label{dual}
0\,\lra\, \Omega^1_{\M_g} \,\lra\, \operatorname{At}^1(\vbE)^*\,\overset{\alpha}{\lra}\,
\mathcal{O}_{\M_g}\,\lra\, 0.
\end{equation}
Let $\boldsymbol{1}_{\M_g}$ be the constant function $1$ on
$\M_g$. Then
\begin{equation}\label{vp}
\Conn(\vbE)\,:=\,\alpha^{-1}(\boldsymbol{1}_{\M_g})\,
\overset{\varphi}{\lra}\, M
\end{equation}
is a holomorphic fiber bundle over $\M_g$ whose fibers are affine
subspaces of the fibers of $\operatorname{At} (E)^*$. From
\eqref{dual} it follows that $\Conn (\vbE)$ is a holomorphic
$\Omega^1_{\M_g}$-torsor. For any open subset $U\,\subset\, \M_g$, the
complex connections on the line bundle $\vbE\big\vert_{U}$ are in
bijection with the $C^{\infty}$ sections of $\Conn(\vbE)$ over $U$
(see \cite{At}). More explicitly, by \cite[Proposition 3.3, Lemma
3.4]{BHS} there is a tautological connection $D$ on $\varphi^*\vbE$,
where $\varphi$ is the projection in \eqref{vp}, and for any complex
connection $\nabla$ on $\vbE$, we have
\begin{gather*}
\nabla\ =\ f_{\nabla}^*D
\end{gather*}
with $f_{\nabla}\,:\, \M_g\,\lra\, \Conn(\vbE)$ the $C^{\infty}$
section of \eqref{vp} associated with $\nabla$.
\end{say}

\section{A correspondence}\label{sec:correspondence}

\begin{say}
In the set-up of Section \ref{connections} substitute
\begin{gather}\label{dhlb}
E\, =\, {\mathcal L}\,:=\, ( \det \pi_* K_{\Cg/\Mg})^*,
\end{gather}
where $\pi$ is the projection in \eqref{pi}. Recall from Section \ref{se2.3} that 
$\Conn(\mathcal L)$ is a $\Omega^1_\Mg$--torsor. In particular, $\Omega^1_\Mg$ acts
on $\Conn(\mathcal L)$. As in \cite[\S~3]{BFPT} we re-scale the $\Omega^1_\Mg$-action on
$\Conn(\mathcal L)$ as follows. Set for simplicity
		\begin{gather}\label{eq:7}
			\mu\,\,:=\,\,\mut .
		\end{gather}
		If
		\begin{gather*}
			A\,:\, \Conn(\mathcal L)\times_{\M_g} \Omega^1_{\M_g} \,\lra\,
			\Conn(\mathcal L)
		\end{gather*}
		is the action of $\Omega^1_{\M_g}$ on $\Conn( \mathcal L)$,
		then define
		\begin{gather}\label{ej}
			A^{\mu}\,:\, \Conn(\mathcal L)\times_{\M_g} \Omega^1_{\M_g}\,\lra\,
			\Conn(\mathcal L), \ \ \ (z,\,v)\,\longmapsto\, A\left (z,\, \mu v
			\right).
		\end{gather}
		We will denote by $\Conn^{t}(\mathcal L)$ the associated rescaled
		torsor; so the total spaces of $\Conn( \mathcal L)$ and $\Conn^{t}(\mathcal L)$
coincide. The curvature of a connection $D$ on $\mathcal L$, which is a
		2-form on $\Mg$, will be denoted by $\erre(D)$. Consider the map
		\begin{gather}
			\Curv \,:\, \Ga (\Conn^t(\mathcal L)) \,\lra\, \forme^{1,1}(\Mg), \qquad
			D\,\longmapsto\, \erre(D)^{1,1}\, ;
			\label{Curv}
		\end{gather}
		the $(1,1)$-component of a 2-form $\alf$ is denoted by $\alf^{1,1}$ (recall
that the two fiber bundles $\Conn( \mathcal L)$ and $\Conn^{t}(\mathcal L)$ are identified,
and only the actions of $\Omega^1_{\M_g}$ are different  --- they differ by multiplication
with the constant $\mu$). Also define
\begin{gather}
  \label{defTT}
  \TT \,:\, \Gamma (\mathcal P_g) \,\lra\, \forme^{1,1}(\Mg), \qquad
  \beta\, \longmapsto\, \muto
  \, \debar \beta,
\end{gather}
where $\debar$ and $\mu$ are defined in \eqref{dv} and \eqref{eq:7} respectively.
	\end{say}

\begin{defin}\label{def:class}
Let $\om_\WP$ denote the K\"ahler form of the Weil-Petersson metric on $\Mg$, and let
		\begin{gather*}
			\class \,\subset\, \forme^{1,1}(\Mg)
		\end{gather*}
		be the space of representatives of the cohomology class
		$\mu\cd [ \om_{\rm WP} ]\,\in\, H_{\debar}^{1,1}(\Mg)$.
	\end{defin}
	
	We recall two results.

\begin{teo}[{\cite[Theorem 3.1]{FPT}}]\label{teo-fpt}
When $g\,\geq\, 5$, the moduli space $\M_g$ does not admit any nonzero
holomorphic $1$-form in the orbifold sense.
\end{teo}

\begin{prop}\label{quillenfamuniv}
Consider the universal family of curves as in \eqref{pi}. Fix on
$\Cg$ the trivial line bundle with the trivial metric.  Endow $\Cg$
with the relative Poincar\'e metric on the fibers. This produces a
Quillen metric on the determinant line bundle
\begin{gather*}
\det \pi_{!}\mathcal{O}_{\Cg}\,\,=\,\, \lambda(\mathcal{O}_{\Cg}) \,\,\cong\, \, {\mathcal L}
\end{gather*}
(see \eqref{dq} for $\lambda(\mathcal{O}_{\Cg})$), where $\pi$ is the
map in \eqref{pi}, whose connection $\nabla^Q$ has curvature
\begin{gather}\label{q1}
\erre(\nabla^Q) \,=\,\mu \, \omega_{\rm WP}.
\end{gather}
\end{prop}

A proof of Proposition \ref{quillenfamuniv} can be found in \cite[p.~184, Theorem 2]{ZT1} with a
different constant due to the difference in normalizations. The constant of \cite{ZT1} was
used in \cite{BFPT}. Here we use a different normalization which gives the
above constant.

\begin{remark}\label{rem-q}
Proposition \ref{quillenfamuniv} also follows easily by combining
\cite[Theorem 0.1]{BGS} and \cite[Corollary 5.11]{wolpert-Chern}. Let $\Phi\, :\, Z\, \longrightarrow\,
B$ be a holomorphic family of compact connected K\"ahler
manifolds and $V\, \longrightarrow\, Z$ a holomorphic Hermitian vector bundle. Then the
holomorphic line bundle $\det \Phi_{!} V$ on $B$ is equipped with a Quillen metric. In this general set-up,
\cite[Theorem 0.1]{BGS} gives a formula for the curvature of the Chern connection for the Quillen
metric on $\det \Phi_{!} V$. Using \cite[Corollary 5.11]{wolpert-Chern}, this general formula in \cite{BGS}
reduces to Proposition \ref{quillenfamuniv} in the special case of the set-up of
Proposition \ref{quillenfamuniv}. (Note that
it should be $\kappa_1 \,=\, \om_\WP /(2 \pi^2)$ in \cite[Corollary 5.11]{wolpert-Chern}
which is evident from the proof.)
\end{remark}
	
\begin{teo}\label{thm:correspondence}
Assume that $g\,\geq\, 5$. The following statements hold:
\begin{enumerate}
\item The maps $\Curv$ and $\TT$ constructed in \eqref{Curv} and
\eqref{defTT} respectively have image $\class$ and are also injective.
\item There is a unique holomorphic isomorphism of
$\Omega_\Mg^1$-torsors
\begin{gather*}
\mathbb{F}\,\,:\,\, \Conn^t(\mathcal L) \,\,\lra\,\, \mathcal P_g.
\end{gather*}

\item Denoting the map of global sections induced by $\mathbb F$ also by $\mathbb F$, the following
diagram is commutative
\begin{equation}\label{eq:4}
\begin{tikzcd}
\Gamma (\Conn^t (\mathcal L)) \arrow[rr, "\mathbb{F}"]
\arrow[dr,swap,"\Curv"]
& & \Gamma (\mathcal P_g)  \arrow[dl,"\TT"] \\
& \class &
\end{tikzcd}
\end{equation}
where $\class$ is as in Definition \ref{def:class}.
			
\item $\mathbb{F} ( \nabla^Q ) \,\,=\,\, \beta^u$, where $\nabla^Q$ is the connection in
Proposition \ref{quillenfamuniv}.

\item
\begin{equation*}
\Curv (\nabla^Q)\, = \, \muto \, \debar \beta^ u  \,=\,\mu \,
\om_{\rm WP},
\end{equation*}
where $\om_{\rm WP}$ is the Weil-Petersson K\"ahler form on $\Mg$.
\end{enumerate}
\end{teo}

\begin{proof}
Consider the section $\nabla^Q\,\in\, \Gamma (\Conn^t (\mathcal L))$
from Proposition \ref{quillenfamuniv}. By \eqref{q1},
$$\Curv (\nabla^Q) \,=\, \erre (\nabla^Q) \,=\, \mu \om_\WP \,\in\, \class .$$
If $D \,\in\, \Gamma (\Conn^t( \mathcal L))$, then $D \,=\, \nabla^Q +\alf$
for some $\alf \,\in\, \forme^{1,0}(\Mg)$. Thus we have
$\erre (D) \,=\, \erre (\nabla^Q) + d\alf$ and
$\Curv (D) \,=\, \Curv (\nabla^Q) + \debar\alf \,\in\, \class$.

To show that the map $\Curv$ surjects to $\class$,
if $\gamma \,\in\, \class$, then for some $\alpha\,\in\, \forme^{1,0}(\Mg)$
we have
\begin{gather*}
\gamma \,=\, \mu \om_{\rm WP} + \debar \alpha .
\end{gather*}
This implies that the connection $D\,=\, \nabla^Q+ \alpha$
has the property that $\Curv ( D)\,=  \Curv(\nabla^Q)+\debar \alpha\,=
\, \gamma$ (see Proposition \ref{quillenfamuniv}). This proves that the map $\Curv$ surjects to $\class$.

Next consider two sections $D_1, \,D_2$ of $\Conn^t (\mathcal L)$ such that
		\begin{gather*}
			\Curv (D_1)\,=\, \Curv(D_2).
		\end{gather*}
		Let $\alpha\,\in\, \forme^{1,0}(\M_g)$ be such that
		$D_2 \,=\, D_1 + \alfa$. Then we have $\erre (D_2) \,=\, \erre (D_1) + d\alfa$ and
		$\Curv(D_2) \,= \, \Curv(D_1) + \debar \alfa.$ Hence we have $\debar \alpha \,=\,0$,
		i.e.,  $\alpha\,\in\, H^0(\M_g,\, \Omega^1_{\M_g})$. By Theorem
		\ref{teo-fpt} we have $\alpha \,=\, 0$ which implies that $D_1\,=\,D_2$. This proves (1) for
		the map $ \Curv $.  The proof for $\TT$ is similar, using that
		$\TT(\beta^u ) \,=\, \mu \om_\WP$, which follows immediately from \eqref{eq:ZTWP}.  This proves (1).
		
		Statements (2), (3), (4) and (5) are explicitly proven in
		\cite{BFPT}. The proof uses Proposition \ref{quillenfamuniv}. We
		remark that if one imposes (4), then one gets a well-defined
		isomorphism of torsors $\mathbb{F}$. The point is that it is
		holomorphic and unique. The latter follows from Theorem
		\ref{teo-fpt}.
	\end{proof}
	
	\section{A Hodge theoretic projective structure}\label{hodgeteo}
	
	\begin{say}
		Let $C$ be a compact connected Riemann surface of genus $g$, with
		$g\, \geq\, 1$. Set $S\,=\, C\times C$, and let $\Delta \,\subset\, S$ be
		the reduced diagonal. In \cite{cpt, cfg}, a form
		\begin{equation}\label{we}
			\widehat{\eta}\,\in\, H^0(S,\, K_S\otimes {\mathcal O}_S(2\Delta))\,=\, H^0(S, K_S(2\Delta))
		\end{equation}
		was constructed in order to describe the second fundamental form of
		the Torelli map. The construction was done using only basic Hodge theory of the
		curve. We briefly recall it. For $x\,\in\, C$,
		consider the homomorphism
		\begin{gather*}
			j_x\,:\, H^0(C,\, K_C(2x))\,\hookrightarrow\,
			H^1(C\setminus\{x\},\, \mathbb{C})\,\cong\,H^1(C,\,
			\mathbb{C})\,=\,H^{1,0}(C)\oplus H^{0,1}(C)
		\end{gather*}
that sends a form to its cohomology class. It is
		injective because $g\, \geq\, 1$. The first isomorphism comes from Mayer-Vietoris sequence.
		Since $H^{1,0}(C)\,\subset\, j_x(H^0(C,\, K_C(2x)))$ and
		$ h^0(C, K_C(2x))\,=\, h^{1,0}(C)+1$, it follows that
		$j_x^{-1}(H^{0,1}(C))\,\subset\, H^0(C,\, K_C(2x))$ is a line.  By
		the adjunction formula we have ${\mathcal O}_C(x)_x \,=\, T_xC$, and hence
		\begin{gather}
			\label{eq:5}
			K_C(2x)_x \,\cong\, (K_C)_x \otimes (T_xC)^{\otimes 2}\,=\, T_xC.
		\end{gather}
		Let
		\begin{equation}\label{ex}
			\eta_x\,\,:\,\, T_xC\,\,\longrightarrow\,\, H^0(C,\, K_C(2x))
		\end{equation}
		be the linear map that sends $v\, \in\, T_xC$ to the unique element in
		the above line $j_x^{-1}(H^{0,1}(C))$ whose evaluation at $x$ coincides,
		via \eqref{eq:5}, with $v$. Let
		$p,\, q\, :\, S\,=\, C\times C\, \longrightarrow\, C$ be the
		projections to the first factor and second factor respectively.
		Define
		\begin{gather*}
			V\,:=\,p_*((q^*K_C)(2\Delta))\ \ \, \text{ and }\ \ \,
			E\,:=\,p_*(K_S(2\Delta)).
		\end{gather*}
		From the projection formula it follows that
		\begin{equation}\label{ep}
			E\,=\,K_C\otimes V.
		\end{equation}
		Note that
		\begin{gather*}
			V_x\,=\, H^0(p^{-1}(x),\, (q^*K_C)(2\Delta)\big\vert_{p^{-1}(x)})\,=\, H^0(C,\,
			K_C(2x)).
\end{gather*}
Hence it follows from \eqref{ep} and \eqref{ex} that
$\eta_x\,\in\, E_x$.  The section $ \eta\,\in\, H^0(C,\, E)$ defined by
$x\,\longmapsto \, \eta_x$ is holomorphic by \cite[Proposition
3.4]{cfg}. Since $E\,=\,p_*(K_S(2\Delta))$, there is a natural
isomorphism $F\,:\, H^0(C,\, E)\,\overset{\sim}{\lra}\, H^0(S,\, K_S(2\Delta))$; define
\begin{gather*}
\widehat{\eta}\,:=\,F(\eta)\,\in\, H^0(S, \,K_S(2\Delta)).
\end{gather*}
Using the notation of Section \ref{cptriemann}, we have
\begin{gather*}
\widehat{\eta}\big\vert_{\Delta_3}\,\,\in\,\,
\widetilde{\Psi}^{-1}(s_{0,C}) \,\,\subset\,\,  H^0\left(\Delta_3,\,\,
K_S(2\Delta)\big\vert_{\Delta_3}\right)
\end{gather*}
(see \eqref{h}). Therefore, via the isomorphism in \eqref{phi}, the
section $\widehat{\eta}\big\vert_{\Delta_3}$ yields a projective
structure on $C$.  The section $\widehat{\eta}$ depends smoothly on the curve
$C$ as it varies in the moduli. This depends on the fact that given a
family of curves both the image of the map $j_x$ and the Hodge
decomposition vary smoothly with the fiber. See 
\cite[Section 3]{BCFP} and in particular \cite[Proposition 3.6]{BCFP}  for more details.
It follows that we get 
a $\cinf$ section
\begin{equation}\label{be}
\beta^{\eta}\, :\, \M_g\, \longrightarrow\, \mathcal{P}_g
\end{equation}
of $\mathcal{P}_g\,\lra \,\M_g$.
Following \cite{BCFP} we call $\beta^\eta$ the
\emph{Hodge theoretic projective structure}.
	\end{say}
	
	\begin{say}
		Let $\A_g$ denote the moduli space of principally polarized complex
		abelian varieties of dimension $g$. Consider the Torelli map
		\begin{gather}\label{emj}
			j\,:\, \M_g \,\lra\, \A_g, \ \ \ [C]\,\longmapsto\,[JC]
		\end{gather}
		which associates to a curve its Jacobian equipped with the natural
		polarization given by the cup product
		$H^1(C,\, {\mathbb Z}) \otimes H^1(C,\, {\mathbb Z})\,
		\longrightarrow\, H^2(C,\, {\mathbb Z})\,=\, {\mathbb Z}$.  The
		moduli space $\A_g$ is the quotient of the Siegel space
		\begin{equation}\label{dss}
	\sieg_g \,:= \,\{\tau \,\in\, \operatorname{Mat}( g\times g,
			\C)\,\,\,\big\vert\,\,\, \tau^t\,=\, \tau ,\ \,
			{\rm Im}\, \tau \,> \,0\}
		\end{equation}
		by the natural action of $\Sp(2g, \Zeta)$ on it. Since the Siegel
		space $\sieg_g$ is a Riemannian symmetric space, it follows that
		$\A_g$ is endowed with a locally symmetric Riemannian metric (in the
		orbifold sense), which is called the \emph{Siegel metric}. The
		Siegel metric on $\A_g$ will be denoted by $\omega_{S}$. In the
		coordinates of $\sieg_g$ we have
		\begin{gather}\label{eq:3}
			\om_S\,:=\, \frac{\sqrt{-1}}{2} ({\rm Im}\, \ta)^{ik} ({\rm Im} \,
			\ta)^{mj} d\ta_{km} \wedge d\overline{\ta}_{ij}.
		\end{gather}
	\end{say}
	
	\begin{teo}[{\cite[Theorem 4.4]{BCFP}}]\label{thmsf}
		With $j$ as in \eqref{emj}, there is $\nu\in \R$, such that
		\begin{gather*}
			\TT( \beta^\eta)\,=\, \muto\, \delbar(\beta^{\eta}) \,\,=\,\, \nu\cd
			j^*\omega_{S},
		\end{gather*}
where $\mu$ is defined in \eqref{eq:7}, and $\beta^{\eta}$ is the section in \eqref{be}.
	\end{teo}
	
	The constant $\nu$ in Theorem \ref{thmsf} will be computed in the next section (see \eqref{multiple}).

	\section{The $L^2$-metric on the Hodge bundle}
	\label{secl2}
	
	In this section we examine the Hodge theoretic projective structure
	$\beta^\eta$ (constructed in \eqref{be}) in the context of
	Theorem \ref{thm:correspondence}. The section of
	$\Conn(\mathcal{L})\,\lra\, \M_g$ corresponding to it in \eqref{eq:4} will be
	constructed.
	
	\begin{say}\label{say:univab}
		Denote by $\Z^g$ the $g\times 1$ matrices (column vectors) with
		entries in $\mathbb Z$, and write
		$\Z^{2g}\,=\, \{(m,\, n)\,\big\vert\,\, m,\, n\, \in\, \Z^g\}$. Let
		$\Z^{2g}$ act on $\siegg\times \C^g$ by the rule
		\begin{gather*}
			(m,\,n) \cd (\tau,\, z) \,\,=\,\, (\tau ,\, z+ m + \tau n ).
		\end{gather*}
The action is holomorphic, free and properly discontinuous. Consider
the corresponding  quotient space
		\begin{equation}\label{uav}
\univabt\,\,:=\,\, \left (\siegg \times \C^g \right) / \Z^{2g}
		\end{equation}
		with the projection
		\begin{gather}
			\label{univab}
			\varpit\,:\, \univabt \,\lra \,\siegg,\, \ \ \, [\tau,\, z] \,\longmapsto\, \tau.
		\end{gather}
		Then $\varpit\meno(\tau) \,=\, \{\tau\} \times A_\tau $, where
		\begin{equation}\label{dat}
			A_\tau\,:= \,\varpit\meno(\tau) \,\cong\, \C^g / \Ga_\tau,\ \ \,
			\Ga_\tau\,:=\, \Z^g + \tau \Z^g.
		\end{equation}
		Note that $A_\tau$ is an abelian variety equipped with the principal polarization
		\begin{gather}\label{defom}
			\om_\tau\,:=\, \frac{\sqrt{-1}}{2} ({\rm Im}\, \tau)^{ij} dz_i \wedge
			d\barz_j,
		\end{gather}
		and the images of $e_1,\, \cdots,\, e_g,\, \tau_1,\, \cdots,\, \tau_g $
		(the columns of $\tau$) give a symplectic basis of
		$H_1(A_\tau, \,\Z)$.  Thus \eqref{univab} is the universal family of
		principally polarized abelian varieties with a symplectic basis.
	\end{say}
	
	\begin{say}
Let
\begin{gather*}
\Omega^1_{\univabt / \sieg_g}\,\, \longrightarrow\,\, \univabt
\end{gather*}
		be the relative cotangent bundle. Then
\begin{equation}\label{de}
\mathcal E\,\,:=\,\, \varpit_* \Omega^1_{\univabt / \sieg_g}\,\, \longrightarrow\,\sieg_g
\end{equation}
is a holomorphic vector bundle on $\sieg_g$ of rank $g$. Recall that $\om_{\tau}$
in \eqref{defom} is a K\"ahler form on $A_\tau$. We fix the corresponding (flat)  K\"ahler
metric on $A_\tau$. Using it we construct the $L^2$-Hermitian product on
$\mathcal{E}_{\tau}\,=\,H^0(A_{\tau},\,\Omega^1_{A_{\tau}})$:
\begin{gather}\label{eq:h}
h_{L^2} ( \alfa,\, \beta ) \,\,:=\,\, \int_{A_{\tau}} \sx \alf,\,{\beta} \xs
\frac{\om_\tau^{g}}{g!},
\end{gather}
where $\sx\ ,\ \xs$ is the Hermitian structure on $T^*A_\tau$ obtained from the
K\"ahler metric on the abelian variety $A_\tau$.
		
		\begin{prop}\label{teorelct} Denote by $\det h_{L^2}$ the hermitian metric induced by
			$h_{L^2}$ on the line bundle $\det \mathcal{E}$, where $\mathcal{E}$
is defined in \eqref{de}. Then
			\begin{gather}
				\label{eq:8}
				\erre(\det \mathcal{E},\, \det h_{L^2})\,\,=\,\, -
				\frac{\sqrt{-1}}{2}\om_S,
			\end{gather}
			where $\omega_{S}$ is the Siegel form (see \eqref{eq:3}).
		\end{prop}
		
		\begin{proof}
			The integrand in \eqref {eq:h} is translation invariant on the torus. Hence
			$$h_{L^2} ( \alfa,\, \beta ) _{L^2}\,=\, \sx \alf,\,{\beta} \xs
			\cd \vol(A_\tau), $$
			and for $ \alf \,=\, \alf_i dz_i$,\, $ \beta \,=\, \beta_j dz_j$,
			\begin{gather*}
				\sx \alf,\, \beta \xs \,\,=\,\,2 ({\rm Im}\, \tau)_{ij} \alf_i \beta_j.
			\end{gather*}
			We consider $dz_1, \lds, dz_g $ as global sections of
			$ \mathcal \varpit_*E \,\lra\, \sieg_g$, so
			$s\,:=\,dz_1 \wedge \cds \wedge dz_g$ is a global section of
			$\det (\varpit_*\mathcal{E})$ over $\sieg_g$, and
			\begin{gather*}
				||s||^2 \,=\, \det (\sx dz_i, dz_j\xs )\,=\, 2^g\det ({\rm Im}\, \tau) \\
				\erre(\det (\varpit_*\mathcal{E}),\, \det h_{L^2}) \,=\, - \de \debar
				\log ||s||^2 \,=\, \frac{1}{4} ({\rm Im}\, \ta)^{ik} ({\rm Im}\, \ta)^{mj}
				d\ta_{km} \wedge d\overline{\ta}_{ij}.
			\end{gather*}
			Recall that the Siegel metric is
			\begin{gather}\label{eq:3b}
				\om_S\,:=\, \frac{\sqrt{-1}}{2} ({\rm Im}\, \ta)^{ik} ({\rm Im}\,
				\ta)^{mj} d\ta_{km} \wedge d\overline{\ta}_{ij}.
			\end{gather}
			Hence
			\begin{gather*}
				\erre(\det (\varpit_*\mathcal{E}),\, \det h_{L^2})\,\,=\,\, -
				\frac{\sqrt{-1}}{2}\om_S.
			\end{gather*}
			This completes the proof.
		\end{proof}
	\end{say}
	
	\begin{say}
		The entire construction done above evidently descends to $\A_g$. This means
		that if we consider the (orbifold) universal family
		$\univab \,\lra\, \A_g$ of abelian varieties, and the Hermitian
		vector bundle
		$$(\varpi_* \Omega^1_{\mathcal A / \A_g} ,\, h_{L^2})\,\lra \,
		\univab ,$$ then
		$\det(\varpi_* \Omega^1_{\univab / \A_g})\,\lra \,\A_g$ equipped
		with the determinant metric $\det(h_{L^2})$ still has the property
		of Proposition \ref{teorelct}.
		
	\end{say}
	
	\begin{say}
		If we pull-back this to $\Mg$ using the Torelli map in \eqref{emj}, then $h_{L^2}$ becomes the usual
		$L^2$-scalar product on $H^0(C,\,K_C)\, =\, H^0(JC,\, \Omega^1_{JC})$:
		\begin{gather*}
			h_{L^2}(\alf,\, \beta) \,=\, \sqrt{-1} \int_C \alf \wedge \bar {\beta},
			\qquad \alf,\, \beta \,\in\, H^0(C,\, K_C).
		\end{gather*}
		We now have the following theorem.
	\end{say}
	
\begin{teo}\label{teoL2}
Let $\mathcal L$ be the dual of the Hodge line bundle as in
\eqref{dhlb}, i.e.,
\begin{gather*}
\mathcal L\,\,:=\,\, j^*(\det \mathcal E)^*.
\end{gather*}
Denote by $h$ the Hermitian metric on $\mathcal L$ dual to
$\det h_{L^2}$. Then the curvature of the Chern connection on $({\mathcal L},\, h)$
is as follows:
\begin{gather}\label{curv}
\erre(\mathcal L,\, h)\,\, =\,\, \frac{\sqrt{-1}}{2} j^*\om_S.
\end{gather}
Denote by $D(\mathcal L,\, h)$ the Chern connection for
$(\mathcal L,\, h)$. Then for $g\,\geq\, 5$, the isomorphism $\mathbb F$
in Theorem \ref{thm:correspondence} satisfies the equation
\begin{gather}\label{maineq}
			{\mathbb F} ( D(\mathcal L,\, h)) \,=\, \beta^\eta,
		\end{gather}
where $\beta^\eta$ is the section in \eqref{be}.
In particular, the constant multiple in Theorem \ref{thmsf} is given
		by
		\begin{gather}\label{multiple}
			\delbar \beta^{\eta}\,\,=\,\, \pi \,j^*\omega_{S}.
		\end{gather}
	\end{teo}
	
	\begin{proof} To prove \eqref{curv} it is enough to dualize
		\eqref{eq:8} and pull it back via the Torelli map $j$ in \eqref{emj}. Since
		$ \erre(\mathcal L,\, h) \,=\, \Curv (\mathcal L,\, h)$, Theorem
		\ref{thm:correspondence} implies that
		\begin{gather*}
			\frac{\sqrt{-1}}{2} j^*\om_S \,\,\in\,\, \Lambda
		\end{gather*}
(see Definition \ref{def:class} for $\Lambda$). Similarly Theorem \ref{thmsf} and
Theorem \ref{thm:correspondence} together imply that
		\begin{gather*}
			\nu\cd j^*\omega_{S} \,\,=\,\, \TT( \beta^\eta) \,\, \in\,\, \Lambda.
		\end{gather*}
Therefore, by the definition of $\Lambda$ (see Definition \ref{def:class}), there is
a $\gamma\,\in\, \forme^{1,0}(\Mg)$ such that
		\begin{gather*}
			\left ( \frac{\sqrt{-1}}{2} -\nu \right) j^*\om_S \,\,=\,\, \debar \ga.
		\end{gather*}
		We claim that $\nu \,=\, \sqrt{-1} /2$. Otherwise $j^*\om_S$ would be $\debar$-exact.
But this is impossible. Indeed assume, to prove by contradiction, that
		\begin{gather*}
			j^*\om_S \,=\, \debar\xi.
		\end{gather*}
We can find a smooth complete curve $C\,\subset\, \Mg$. This can be seen as follows: Consider the
Satake compactification of the moduli space $\A_g$ of principally polarized abelian varieties
(see \cite{Sa1} and \cite{Sa2}). Recall the Torelli map $j$ in \eqref{emj}. Let
$\overline{\M}^S_g$ be the closure of the image $j(\M_g)$. The codimension of the complement
$\overline{\M}^S_g\setminus \M_g \, \subset\, \overline{\M}^S_g$ is at least two. Consequently,
there are smooth complete curves $C$ in the projective variety $\overline{\M}^S_g$ that are contained
in the Zariski open subset $\M_g\, \subset\, \overline{\M}^S_g$.

Since $j(C)$ is a curve in $\A_g$, we would have
		\begin{gather*}
			0 \,<\, \int_C j^*\om_S\, =\, \int_C\debar \xi \,=\, \int_C d\xi\, =\,0.
		\end{gather*}
In view of this contradiction we conclude that $j^*\om_S$ is not $\debar$-exact. Consequently,
we have $\nu \,=\, \sqrt{-1} /2$ which proves the claim.
		
		So
		\begin{gather}
			\label{constant}
			\TT(\beta^\eta) \,\,= \,\,\frac{\sqrt{-1}}{2} j^*\om_S .
		\end{gather}
		This is equivalent to \eqref{multiple}.

{}From \eqref{constant} and
		\eqref{curv} it follows that
		$\Curv(D(\mathcal L,\,h)) \,= \,\TT(\beta^\eta)$.  Using Theorem
		\ref{thm:correspondence} we immediately get \eqref{maineq}.
	\end{proof}
	
	\section{Quillen metric and Theta line bundle}\label{sec:abelian-varieties}
	
	We now give an alternative construction of the section of
	$\Conn(\mathcal{L})\,\lra\, \M_g$ corresponding to the Hodge theoretic
	projective structure $\beta^{\eta}$. Like
	the characterization in Theorem \ref{thm:correspondence} of the
	projective structure given by uniformization, it also uses the machinery
	of the determinant of the cohomology and Quillen metric.
	
	\begin{say}\label{Detandquill}
		We start by recalling the definitions of a determinant bundle and
		Quillen metric; see \cite{BGS} and \cite{soule-Arakelov} for more
		details (some sign conventions in the latter reference are different
		from ours).
		
		Let $X$ be a compact connected K\"ahler manifold of complex dimension $d$. Fix
		a K\"ahler form $\om$ on $X$. Take a holomorphic vector bundle
		$E\, \longrightarrow\, X$ equipped with a Hermitian metric
		$h$. These data give the following:
		\begin{itemize}
			\item For all $0\, \leq\, i\, \leq\, d$, an $L^2$-metric $\scalo$ on
			$\forme^{0,i}(X,\,E)\,:=\, \cinf(X,\, \est^{0,i}T^*X\otimes E)$.
			
			\item Nonnegative self-adjoint elliptic operators
			$\Delta_{E,i} \,:=\, \debar_E \debar^*_E + \debar_E^* \debar_E$
			acting on $\forme^{0,i}(X,\,E)$.
			
			\item Isomorphisms
			$H^i(X,\,E) \,\cong\, \harm^{0,i}(X,\,E) \,:= \,\ker
			\Delta_{E,i}$.
			
			\item A countable sequence of eigenvalues of $\Delta_{E,i}$
\begin{gather*}  
0\,\leq\, \lam_{i,1} \,\leq \,\lam_{i,2} \,\leq \,\cds \,\leq\, \lam_{i,n} \,\leq\, \cds
\end{gather*}
repeating according to multiplicities.
			
			\item A meromorphic function $\zeta_{E,i}$ on $\C$ such that
			\begin{gather*}
				\zeta_{E,i}(s) \,=\, \sum^\infty_{
					\begin{subarray}{c}
						n=1\\ \la_{i,n}\neq 0
					\end{subarray}
				} \la_{i,n}^{-s}
			\end{gather*}
			for $\Re (s) \,>\, d$ and which is also holomorphic at $s\,=\,0$.
		\end{itemize}
		
		Consider the one-dimensional complex vector space
		\begin{gather*}
			\la(E) \,\,: =\,\, \bigotimes_{i=0}^d ( \det
			H^i(X,E))^{(-1)^{i+1}}.
		\end{gather*}
		The $L^2$-metrics on $\forme^{0,i}(X,\,E)$,\ $i\, \geq\, 0$, induce
		metrics on $\harm^i(X,\,E)$, $H^i(X,\,E)$, $\det H^i(X,\,E)$ and
		$\la(E)$. Let $h_{L^2}$ denote the metric thus obtained on
		$\la(E)$. The \emph{analytic torsion} of $(X,\,\om,\, E,\, h)$ is
		defined to be
		\begin{equation}\label{eat}
			T(X,\,\om,\, E,\, h)\,\,:=\,\, \sum_{i=0}^d (-1)^{i} i \zeta'_{E,i} (0).
		\end{equation}
		The \emph{Quillen metric} for the quadruple $(X,\,\om,\, E,\, h)$ is
		the Hermitian metric
		\begin{gather}
			\label{defq}
			h_Q(X,\om, E, h)\,\,:=\,\, e^{T(X,\om, E,h)} \cd h_{L^2}
		\end{gather}
		on the complex line $\la(E)$ which depends on $X,\,\om,\, h$.
		
		Assume that we have a smooth holomorphic family of K\"ahler
		manifolds
		\begin{gather*}
			\pi\,:\, \mathscr{X} \,\lra\, B
		\end{gather*}
		with connected fibers and smooth base $B$. Let ${E} \,\lra\, \mathscr{X}$ be a holomorphic vector bundle
		with a Hermitian metric $h$. Let $\om$ be a closed real $(1,\,1)$--form on $\mathscr{X}$ which is nonnegative. 
		For any point
		$t\,\in\, B$ set
		\begin{gather*}
			X_t\,:=\,\pi\meno (t),\ \ \, E_t\,:=\, {E}\restr{X_t}, \ \ \,
			h_t\,:=\, h\restr{E_t}, \ \ \, \om_t\,:=\,\om \restr{X_t}.
		\end{gather*}
		Assume that $ \om_t$ is a K\"ahler form
		on $X_t$ for every $t$.  Then the 1-dimensional vector spaces
		$\la(E_t)$,\, $t\,\in \,B$, glue together to form a holomorphic line
		bundle
\begin{equation}\label{dq}
\la(E) \,\lra\, B , \ \ \ \, \la(E)_t\,:=\, \la(E_t),
\end{equation}
which is known as the \emph{determinant line bundle} or
		\emph{determinant of the cohomology}.  Moreover the above Hermitian
		metrics
		\begin{gather*}
			h_{Q,t}\,\,:= \,\, h_Q (X_t,\,\om_t,\, E_t,\, h_t)
		\end{gather*}
		on $\la(E_t)$ glue together to yield a smooth hermitian metric on
		$\la(E)$, which is known as the \emph{Quillen metric}.
	\end{say}
	
	\begin{say}
Let us consider again the universal abelian variety in \eqref{univab} over the Siegel space.
		The fiber over any $\tau \,\in \,\sieg_g$ is the abelian variety $A_\tau$.
		Holomorphic line bundles on $A_\tau$ can be described via factors of
		automorphy \cite[p.~58ff]{igusa}, which are holomorphic functions
		\begin{gather*}
			e\,\,: \,\,\Ga_\tau \times \C^g \,\,\lra\,\, \C^*
		\end{gather*}
		(see \eqref{dat}) satisfying the condition
		\begin{gather*}
			e_{\ga + \ga'} (z) \,=\, e_\ga (z+\ga') \cd e_{\ga'}(z).
		\end{gather*}
		A factor of automorphy produces an action of $\Ga_\tau$ on
		$\C^g \times \C$ as follows:
		\begin{gather*}
			\ga \cd (z,\, t)\,\,:=\,\, (z+ \ga ,\,\, e_\ga (z) \cd t).
		\end{gather*}
		The corresponding quotient
		\begin{gather*}
			\mathbf{L}_e \,\,=\,\,(\C^g \times \C)\,\big /\,{\Ga_\tau} \,\,\lra\,\,
			A_\tau
		\end{gather*}
		is a holomorphic line bundle.
	\end{say}
	
	\begin{say}
		By the Appell--Humbert theorem, all line bundles on $A_\tau$ arise this
		way. Nevertheless different factors of automorphy can give rise to
		the same line bundle; two factors of automorphy giving isomorphic
		line bundles are called equivalent. There are at least two ways to
		choose a preferred factor of automorphy in an equivalence class.
		One choice corresponds to the factors that are called
		\emph{canonical} \cite{lange-birkenhake} or \emph{normalized} \cite
		{igusa, debarre}. These are congenial for description of Hermitian
		metrics on the line bundles. Another choice is given by the
		\emph{classical} factors.  These vary holomorphically with $\tau$
		which makes them more useful when the abelian variety is moving in a
		holomorphic family. The holomorphic sections of a line bundle
		described by a factor of automorphy are represented by the theta
		functions which depend on the factor of automorphy. When the factor
		of automorphy changes, the theta functions change
		accordingly. Therefore the same bundle can be described via
		normalized or classical factors of automorphy, and its sections can
		be described by either normalized or classical theta functions.
		
	\end{say}
	
	\begin{say}
		
		In the following we will be interested in only one particular line bundle. We
		will need its classical expression to establish holomorphic
		dependence on the moduli. To deal with its Hermitian metric we will rely
		on a computation done in terms of the normalized description (see
		Proposition \ref{elle2}).
		
		For $z,\, w \,\in \,\C^g$, set
		\begin{gather*}
			\sx z, \,w \xs \,\,:=\,\, \sum_{i=1}^g z_i w_i.
		\end{gather*}
For $\tau\, \in\, \sieg_g$ (see \eqref{dss}), fix the following factor of automorphy:
		\begin{gather*} e\,:\, \Ga_\tau \times \C^g \,\lra\, \C^*, \ \ \,
			e_\ga(z)\,:=\, \exp ({-\pi\sqrt{-1}\sx n,\, 2z + \tau n\xs })
		\end{gather*}
		for $\ga \,= \,m + \tau n$ and $m,\, n \,\in\, \Z^g$ (see
		\eqref{dat} for $\Ga_\tau$); it is a classical automorphy factor. Let
		\begin{equation}\label{ltau}
			L_\tau\,\, \longrightarrow\,\, A_\tau
		\end{equation}
		be the corresponding line bundle. The Riemann theta function
		\begin{gather*}
			\vartheta_\tau(z)\,:=\, \vartheta
			\begin{bmatrix}
				0 \\ 0
			\end{bmatrix}
			(z, \tau) \,=\, \sum_{m \in \Z^g} e^{\pi\sqrt{-1} \sx m, 2z + \tau
				m \xs } .
		\end{gather*}
		represents a nonzero element of $H^0(A_\tau,\, L_\tau)$ (which is
		one-dimensional). This $\theta_\tau$ is an even function, and hence
		$L_\tau$ is a \emph{symmetric line bundle}   \cite[p.~34]{lange-birkenhake}. Set
		\begin{gather*}
			H\,:\, \C^g \times \C^g \,\lra\, \C, \\
			H(z,\,w)\,:=\, \sqrt{-1} \om(z,\,w)+\om (z,\, \sqrt{-1}w) \,=\,
			({\rm Im}\, \tau)^{ij} \overline{z}_i w_j.
		\end{gather*}
		So $H$ is $\mathbb C$-linear in the second argument and conjugate linear
		in the first. Define a Hermitian metric $h$ on the trivial line bundle
		$\C^g \times \C\, \longrightarrow\, \C^g$ by the rule
		\begin{gather}\label{defh}
			||(z, \,t)||_h \,\,:=\,\, |t| \cd \exp ({-\pi H(y,y)}),
		\end{gather}
		where $z\,=\,x+\sqrt{-1} y$. This metric is $\Ga_\tau$-invariant,
		and hence it descends to a metric $h$ on the line bundle $L_\tau$ in \eqref{ltau}. Since
		\begin{gather*}
			H(y,\,y) \,=\, \frac{1}{4} H(z-\barz, \,z - \barz ) \,=\,
			\frac{1}{2} H(z,\,z) - \frac{1}{4} ({\rm Im}\, \tau)^{ij} (z_i z_j +
			\barz_i \barz_j ) ,
		\end{gather*}
		and $\de\debar z_iz_j \,=\,0 \,=\, \de\debar \barz_i\barz_j$, we
		conclude that
\begin{gather}
\notag \de \debar H(y,\,y) \,=\,\frac{1}{2} \de \debar H(z,\,z),\\
\label{ciuno} c_1(L,\,h) \,=\, -\frac{\sqrt{-1}}{2\pi} \de \debar
\log ||s||_h^2 \,=\, \sqrt{-1} \de \debar H(y,\,y) \,=\,
\frac{\sqrt{-1}}{2} \de \debar H(z,\,z) \,=\, \om.
\end{gather}
The line bundles $L_\tau$ in \eqref{ltau}, $\tau\, \in\, \sieg_g$ (see \eqref{dss}), together define
a holomorphic line bundle on $\univabt$ (see \eqref{uav}). The group $\Z^{2g}$ has the following
free and properly discontinuous holomorphic action on $\sieg_g \times \C^g \times \C$:
\begin{gather*}
(m,\,n) \cd (\tau,\, z,\,t) \,\,:=\,\, (\tau ,\, z+ m + \tau n,\,\,
\exp({\pi\sqrt{-1}\sx n,\,\, 2z + \tau n\xs })\cd t).
\end{gather*}
Then
\begin{equation}\label{dt}
\LB \,\,: =\,\, ( \siegg\times \C^g \times \C ) \big / \Z^{2g}
\,\lra\, (\siegg \times \C^g ) \big / \Z^{2g} \,=\, \univabt
\end{equation}
is a holomorphic line bundle. Note that $\varpi\meno(\tau)\,=\, A_\tau$ and
$\LB\restr{\varpi\meno(\tau)} \,=\, L_\tau$. Moreover the Hermitian
metric $h$ in \eqref{defh} varies smoothly with $\tau$ and is left
invariant by the action of $\Z^{2g}$. Hence it produces a global Hermitian metric
on $\LB$.

The line bundle $\LB$ is not $\Sp(2g,\Z)$ invariant, and therefore it does
not descend to $\A_g$. However, the following proposition holds.
\end{say}

\begin{prop}[{\cite[p.~80, Lemma 7]{igusa}}]\label{elle2}
		Consider $L_\tau^{\otimes 2}$ endowed with the metric
		$h^{\otimes 2}$ (see \eqref{defh}). Construct $L^2$-metric on
		$H^0(A_\tau,\,L^{\otimes 2}_\tau)$ using $h^{\otimes 2}$ and
		$\om_\tau$ (see \eqref{defom}). Then there is a basis
		$\{\theta_{\tau,i}\}_{i=1}^{2^g}$ of
		$H^0(A_\tau,\,L^{\otimes 2}_\tau)$ which depends holomorphically on
		$\tau$ and which consists of $2^g $ pairwise orthogonal vectors with
		\begin{gather}
			\label{eq:2}
			||\theta_{\tau,i}||^2_{h^{\otimes 2}} \,\,= \,\,(\det {\rm Im}\,
			\ta)^{-1/2} .
		\end{gather}
	\end{prop}
	
	We wish to apply the methods of analytic torsion and Quillen metric to
	the family $\varpi \,:\, \univabt \,\lra \,\siegg$ and the line bundle
	$\LBD \,\lra\, \univabt$ equipped with the metric $h^{\otimes 2}$.  Fix
	the section $\omega$ in \eqref{defom} as the relative K\"ahler form on
	$\univabt$ (for the projection $\varpi$ in \eqref{univab}). The corresponding Quillen
	metric on $\la(\LBD)$ was computed in \cite{bost}.
	
	\begin{teo}[{\cite[Proposition 4.2]{bost}}]\label{th-b}
		Let $A$ be a $g$-dimensional principally polarized complex abelian
		variety. Let $\om$ be a flat K\"ahler form on $A$. Let
		$L \,\lra\, A$ be an ample line bundle, and let $h$ be a Hermitian
		metric on $A$. Assume that $c_1(L,h)\,:=\, \sqrt{-1} R(h) / 2\pi$ is
		translation invariant.  For a 2-form $\alf$ on $A$, set
		\begin{equation}\label{ral}
			\rho(\alf)\,\,:=\, \,\frac{1}{g!} \int_A \alf^g.
		\end{equation}
		Then the corresponding analytic torsion (see \eqref{eat}) has the
		following expression:
		\begin{gather*}
			T(A,\,\om,\,L,\,h) \,\,=\,\, - \frac{1}{2} \rho(c_1(L,\,h)) \log
			\frac{ \rho(c_1(L,\,h)) } { (2\pi)^g \rho (\om)}.
		\end{gather*}
	\end{teo}
	
(There is a change of sign in the above expression because the
	determinant line bundle defined here is dual of the determinant line
	bundle in \cite{bost}.)
	
	The Quillen metric can be computed using Theorem \ref{th-b}. Let
	$\omega$ denote the relative K\"ahler form on
	$\varpi \,:\, \univabt \,\lra\, \siegg$ in \eqref{univab} constructed
	using $\omega_\tau$ in \eqref{defom}.
	
	\begin{teo}\label{thmabel}
          Let $\la(\LBD)\,\lra\, \siegg$ denote the determinant line
          bundle for the family
          $\varpi \,:\, \univabt \,\lra\, \siegg$ in \eqref{univab}
          and the vector bundle $\Theta^{\otimes 2} \lra \univabt$.
          Let $h_Q$ denote the Quillen metric on $\la(\LBD)$
          constructed using the Hermitian metric $h^{\otimes 2}$ (see
          \eqref{defh}) on $\LBD$ and the relative K\"ahler form
          $\omega$. Then
		\begin{gather}
			\label{eq:1}
			\la(\LBD) \,=\, (\det \varpi_*\LBD)^*,\\
			\label{tors2}
			T (A_\tau,\, \om,\, L_\tau,\, h) \,=\, \frac{1}{2} \log (2\pi)^g,
			\\
			\label{QuiL2}
			h_Q\,=\, (2\pi )^{g/2} \cd h_{L^2},
			\\
			\label{eq:2b}
			\sqrt{-1} \cd \erre(\la( \LB^{\otimes 2}), \,h_Q) \,=\, \sqrt{-1}
			\cd \erre(\la( \LB^{\otimes 2}), \,h_{L^2}) \,=\, 2^{g-2} \,
			\omega_S,
		\end{gather}
		where $\omega_{S}$ is the Siegel form (see \eqref{eq:3}).
	\end{teo}
	
	\begin{proof}
		Since $c_1(L^{\otimes 2}_\tau,\,h^{\otimes 2}) \,= \,2\om_\tau$, the
		line bundle $L_\tau^{\otimes 2}$ is ample and
		$$\dim H^0(A_\ta,\, L^\otwo_\ta) \,=\,\operatorname{Pf}(2\om)
		\,=\,2^g.$$ Therefore, by Kodaira vanishing theorem,
		$H^i(A_\tau, \,L^\otwo_\ta)\,=\,0$ for all $i\,>\,0$. This proves
		\eqref{eq:1}. Moreover, $\rho(\om) \,=\, 1$ (see \eqref{ral}), and
both \eqref{tors2} and \eqref{QuiL2} follow immediately. In particular,
		the analytic torsion is a constant independent of $\tau$, and also
		$h_Q$ is a constant multiple of the $L^2$-metric. Thus the two
		metrics, namely $h_Q$ and the $L^2$-metric, have the same curvature.
		
		We will compute the $L^2$-metric. Set
		\begin{gather}
			v(\tau) \,\,:=\,\, \theta_{\tau,1}\wedge \cds \wedge
			\theta_{\tau,2^g} . \label{defv}
		\end{gather}
		Then $v$ is a holomorphic section of $\det \varpi_* \LBD $.  Let
		$v^* \in C^\infty (\la (\LBD))$ be the section dual to $v$.  By
		Proposition \ref{elle2},
		\begin{gather*}
			||v||_{L^2} ^2 \,=\, (\det {\rm Im}\, \tau)^{-2^{g-1}} \ \ \ \text{ and
			}\ \ \ \quad ||v^*||_{L^2} ^2 \,=\, (\det{\rm Im}\, \tau)^{2^{g-1}}.
		\end{gather*}
		Hence for its curvature we get that
\begin{gather*}
			R(\la(\LBD),\, h_Q) \,=\, R( \la(\LBD),\, || \cd ||_{L^2}) \,=\,
			-\de\debar
			\log || v^*||_{L^2} ^2 \,= \\
			=\,-2^{g-1} \cd\de \debar \log \det ({\rm Im}\, \ta) \,=\, -2^{g-1}\cd
			\de
			\Bigl [ ({\rm Im}\, \ta)^{ij} \debar \,({\rm Im}\, \ta)_{ij} \Bigr ].
\end{gather*}
We have $$\debar \, ({\rm Im}\, \ta_{ij}) \,=\, \frac{\sqrt{-1}}{2}
			d\overline{\ta}_{ij}, \qquad \de \, ({\rm Im}\, \ta_{ij}) \,=\,
			-\frac{\sqrt{-1}}{2} d{\ta}_{ij}.$$
Therefore, it follows that
$$
\de \Bigl [ ({\rm Im}\, \ta)^{ij} \debar \,{\rm Im}\, \ta_{ij} )\bigr ]\,=\,
\de \Bigl [ \frac{\sqrt{-1}}{2} ({\rm Im}\, \ta)^{ij} d \bar{\ta}_{ij }
\Bigr ]\,=\, -\frac{1}{4} ({\rm Im}\, \ta)^{ik} ({\rm Im}\, \ta)^{mj}
d\ta_{km} \wedge d\overline{\ta}_{ij}.
$$
Hence we have
$$
\sqrt{-1} \cd R(\la(\LBD),\, h_Q) \,=\, 2^{g-3} \sqrt{-1} \cd
({\rm Im}\, \ta)^{ik} ({\rm Im}\, \ta)^{mj} d\ta_{km} \wedge
d\overline{\ta}_{ij} .
$$
		Now using \eqref{eq:3} we get \eqref{eq:2b}.
	\end{proof}
	
	\begin{say}
		The line bundle $L^\otwo _\tau$ is canonically associated to the
		abelian variety $A_\tau$. Thus the entire construction descends to
		$\A_g$.  This means that if we consider the (orbifold) universal
		family $\mathcal{A}\,\lra\, \A_g$ of abelian varieties, the
		Hermitian line bundle
		$(\mathcal{L}^{\otimes2},\, h^{\otimes 2})\,\lra \, \mathcal{A}$,
		together with the K\"ahler form $\omega_{\tau}$ on every abelian
		variety $A_{\tau}$, produces the line bundle
		$\la(\LB^{\otimes 2})\,\lra \,\A_g$ equipped with the Quillen metric
		$h_Q$. Also the form $\om_S$ descends to an orbifold K\"ahler metric
		on $\A_g$. And of course the properties \eqref{eq:1} and
		\eqref{eq:2b} still continue to hold at the level of $\A_g$.
	\end{say}

	\begin{say}
		Consider the Hodge theoretic projective structure of Section \ref{hodgeteo}.
		The section of $\Conn(\mathcal{J})\,\lra\, \M_g$ that corresponds to it by
		Theorem \ref{thm:correspondence} will be constructed.
		Let $\mathcal{J}\,\lra \, \M_g$ and $\Cg\,\lra \, \M_g$ be the
		(orbifold) universal Jacobian variety and the (orbifold) universal
		curve respectively:
		\begin{equation}\label{ud}
			\begin{tikzcd}
				\mathcal{J}\,=\,j^*\mathcal{A} \arrow[r, "J"] \arrow[d,swap]
				& \mathcal{A} \arrow[d] \\
				\M_g \arrow[r, "j"] & \A_g,
			\end{tikzcd}
			\qquad \quad
			\begin{tikzcd}
				\Cg \arrow[d,"\pi\,"]
				\\
				\M_g .
			\end{tikzcd}
		\end{equation}
The determinant of cohomology behaves well with respect to base
change, \cite{BGS}, in the following sense: Let $\Phi\,:\, Z\, \longrightarrow\, B$ be a holomorphic family of
compact connected K\"ahler manifolds, and let $V$ be a holomorphic vector bundle on $Z$. Let
$\beta\, :\, B'\, \longrightarrow\, B$ be a holomorphic map. Then we have the family
$$\widehat{\Phi}\, :\, B'\times_B Z\, \longrightarrow\, B'$$ and the vector bundle
$E'\,=\, \widehat{\beta}^*E\, \longrightarrow\,
B'\times_B Z$, where $\widehat{\beta}\, :\, B'\times_B Z\, \longrightarrow\, Z$ is the natural map.
Then the determinant line bundle $\det \widehat{\Phi}_{!} \widehat{\beta}^*E\,
\longrightarrow\, B'$ has a natural holomorphic isomorphism with
$\beta^*\det {\Phi}_{!} E\, \longrightarrow\, B'$. The same is true for the Quillen metric; see
\cite{BGS}. Thus using \eqref{QuiL2} we get the following:
		\begin{gather*}
			\la (j^*\LB^\otwo)  \,=\, j^* \la ( \LB^\otwo) \\
			h_Q (j^*\LB^\otwo) \, =\, j^* h_Q ( \LB^\otwo) \,=\, j^* h_{L^2}
		\end{gather*}
		\begin{gather}\label{omegasiegel}
			\sqrt{-1} \cd \erre(j^*\la( \LB^{\otimes 2}), \,j^*h_Q) \,=\,
			\sqrt{-1} \cd \erre(j^*\la( \LB^{\otimes 2}), \,j^*h_{L^2}) \,=\,
			2^{g-2} \, j^*\omega_S.
		\end{gather}
	\end{say}
	
	We now wish to apply the following result by Kouvidakis.
	
	\begin{teo}[{Kouvidakis}]\label{TK}
		Let  $\mathcal{K}_{\Cg/\Mg}$ denote the relative canonical bundle on $\Cg$, and let $\mathcal{L}$ and $\Theta$
		be the line bundles constructed in \eqref{dhlb} and \eqref{dt} respectively. Then
		\begin{gather*}
			j^* ( \det \varpi_* \LB^{\otimes 2} ) \,\cong\, ( \det \pi_* {\mathcal
				K})^{\otimes (-2^{g-1})} \,=\, \mathcal L ^{ \otimes 2^{g-1}},
		\end{gather*}
		where $j$ and $\pi$ are the maps in \eqref{ud} while $\varpi$ is the map in \eqref{univab}.
	\end{teo}
	
	See \cite[p.~2567, Corollary 4.1]{kouvidakis}. The
	theorem also follows from \cite[Ch.~1, Theorem 5.1]{faltings-chai}.

	Theorem \ref{TK} together with \eqref{eq:1} and \eqref{omegasiegel} implies that
	\begin{gather*}
		2^{g-2} \, j^*\omega_S\,=\, \sqrt{-1} \cd \erre(j^*\la( \LB^{\otimes 2}), \,j^*h_Q)
		\,=\,\sqrt{-1} \cd \erre(j^*(\det \varpi_*\LBD)^*, \,j^*h_Q)\\
		=\, -\sqrt{-1} \cd \erre(j^*(\det \varpi_*\LBD),
		\,(j^*h_Q)^*)\,=\,-\sqrt{-1} \cd \erre( \, \mathcal L ^{ \otimes
			2^{g-1}}, (j^*h_Q)^*).
	\end{gather*}
	In other words, $(j^*h_Q)^*$ is a metric on
	$\mathcal L ^{ \otimes 2^{g-1}}$ such that
	\begin{gather*}
		\erre(\mathcal L ^{ \otimes 2^{g-1}},\, (j^*h_Q)^*) \,=\,2^{g-2} \,
		\sqrt{-1}\, j^*\omega_S.
	\end{gather*}
	If we take $h'$ to be the $2^{g-1}$-th root of $(j^*h_Q)^*$, we get that
	\begin{gather*}
		\erre(\mathcal L, \,h') \,=\,\frac{\sqrt{-1}}{2} \, j^*\omega_S.
	\end{gather*}
	
	Hence, we have proved the following:

	\begin{teo}\label{teo1}
		Let $\mathcal L$ be the Hodge line bundle as in \eqref{dhlb}. Denote
		by $h'$ the Hermitian metric on $\mathcal L$ given by the
		$2^{g-1}$-th root of the hermitian metric $(j^*h_Q)^*$ on
		$\mathcal L ^{ \otimes 2^{g-1}}$. Then
		\begin{gather*}
			\erre (\mathcal L, \,h')\,\, =\,\, \frac{\sqrt{-1}}{2} j^*\omega_S.
		\end{gather*}
		It follows that $D(\mathcal L, \,h')\, = \,D ( \mathcal L,\, h)$. In particular if
$g\,\geq\, 5$, then under
		the correspondence
		\begin{equation*}
			\begin{tikzcd}
				\Gamma (\Conn^t (\mathcal L)) \arrow[rr, "\mathbb{F}"]
				\arrow[dr,swap,"\Curv"]
				& & \Gamma (\mathcal P_g)  \arrow[dl,"\TT"] \\
				& \class &
			\end{tikzcd}
		\end{equation*}
		in Theorem \ref{thm:correspondence}, the Chern connection $D(\mathcal L,\, h')$ is mapped to the
		projective structure $\beta^\eta$.
	\end{teo}
	
	\begin{say}
		In other words, in Theorems \ref{teo1}, and \ref{teoL2} we
		constructed Hermitian metrics $h$ and $h'$ on the Hodge line bundle $\mathcal{L}$ such that
		\begin{gather}\label{curvature}
			R(\mathcal L,\, h)\,\,=\,\,R(\mathcal L, \,h')\,\, =\,\, \frac{\sqrt{-1}}{2} j^*\om_S.
		\end{gather}
		To summarize, we recall that $h$ is defined as the dual of the determinant of the $L^2$-metric on the
Hermitian vector bundle $\varpi_* \Omega^1_{\mathcal A / \A_g}$, while $h'$ is defined as the $2^{g-1}$-th
root of (the pull-back via the Torelli map of) the Quillen metric coming from the Theta line bundle
		endowed with the Hermitian  metric given by the polarization. It follows from \eqref{curvature}
		and Theorem \ref{thm:correspondence} that the Chern connections associated with $(\mathcal L, \,h)$
		and $(\mathcal L,\, h')$ coincide.
		
	\end{say}
	
	\section{Curves and the Arakelov metric}\label{sec:curv-arak-metr}
	
	\begin{say}
		We recall the definition of the Arakelov metric.
		
		Let $\Abel_{x_0}\,:\, C \,\lra\, JC$ be the Abel-Jacobi map with
		base point $x_0\,\in\, C$. Fix a symplectic basis of
		$H_1(C,\,\Z)$, and let $\tau$ be the associated period matrix. The
		polarization $\om_\tau$ in \eqref{defom} is a translation-invariant
		K\"ahler metric on $JC\,=\,A_\tau$. Since $\Abel_{x_0}$ is an
		embedding, the pull-back $\Abel_{x_0}^*\om_\tau$ is a K\"ahler
		metric on $C$. Moreover, it is actually independent of the choice of
		$x_0$ because $\om_\tau$ is translation-invariant, and a different
		choice of the base point results in just translating the map
		$\Abel_{x_0}$. The rescaled pull-back form
		\begin{gather*}
			\mu\,\,:= \,\,\frac{1}{g} \Abel_{x_0}^*\om_\tau
		\end{gather*}
		is called the \emph {canonical metric}.
		
		Note that by \cite[p.~335]{griffiths-harris},
		\cite[p. 97]{narasimhan-Riemann},
		\begin{gather*}
			\int_C \Abelz^* c_1(L_\tau) \,=\,g
		\end{gather*}
		(see \eqref{ltau}). Hence we have $\int_C \mu \,=\,1$. In other words,
		$[\mu]$ is the positive generator of $H^2(C,\,\Z)$.
		
		Now let $h$ be a K\"ahler metric on $C$. Since $\dim _\C C =1$, any
		Hermitian metric on the tangent bundle of $C$ is K\"ahler. So $h$ is
		simply a Hermitian metric on the line bundle $TC$. Its curvature
		satisfies the condition
		\begin{gather*}
			\biggl [ \frac{\sqrt{-1}}{2\pi}R(h) \biggr] \,=\, c_1 (TC) \,=\,
			(2-2\cdot{\rm genus}(C))[\mu]\,=\,(2-2g)[\mu].
		\end{gather*}
		Consequently, there is a function $\chi \,\in\, \cinf(C,\,\R)$ such that
		\begin{gather*}
			\frac{ \sqrt{-1} R(h ) - \sqrt{-1}\de\debar \chi}{2\pi
			}\,\,=\,\,(2-2g)\mu.
		\end{gather*}
		Now recall that if we re-scale the metric $h$ by a factor $e^\psi$
		with $\psi \,\in\, \cinf(C,\,\R)$, then
		\begin{gather}
			\label{riscalo}
			\erre(e^{\psi}h) \,=\, \erre(h) - \de\debar \psi.
		\end{gather}
		Hence the metric $\overline{h}\,:=\,e^\chi h$ satisfies the equation
		\begin{gather}
			\label{firstAra}
			\frac{\sqrt{-1}}{2\pi} \erre(\overline{h}) \,\,=\,\, (2 -2g ) \mu,
		\end{gather}
		in particular, its curvature is a multiple of the canonical
		metric. By \eqref{riscalo}, the condition in \eqref{firstAra}
		identifies $\overline{h}$ up to constant re-scalings, meaning,
		$\overline{h}$ is defined up to multiplication by a positive
		constant. The Arakelov metric $h_{\rm Ar}$ is defined by
		\eqref{firstAra} together with a second condition that normalizes
		this positive constant by comparing $\om_{\rm Ar}$ with the Green
		function of the canonical metric. See e.g. \cite[p. 432]{wentworth}
		for more details.
		
	\end{say}
	
	\begin{say}
		
		Now we consider the relative canonical bundle $\mathcal{K} \lra \Cg$
		equipped with the dual Arakelov metric (denoted $h_{\rm Ar}$) and
		also the family $\Cg\lra \Mg$ with the relative K\"ahler form
		obtained by assigning the Arakelov metric $\om_{\rm Ar}$ on the
		fibers.  Then Faltings' delta invariant is defined as the analytic
		torsion:
		\begin{gather*}
			\delta\,=\,T(\Cg,\, \om_{\rm Ar} ,\, \mathcal{K},\, h_{\rm Ar}).
		\end{gather*}
Now the determinant bundle is $\mathcal L$ and
		\begin{gather*}
			h_Q(\Cg,\, \om_{\rm Ar} ,\, \mathcal{K},\, h_{\rm Ar}) \,= \,\exp(
			{T (\Cg,\, \om_{\rm Ar} ,\, \mathcal{K},\, h_{\rm Ar})}) h_{L^2} =
			e^\delta \, h,
		\end{gather*}
		where $h$ is the $L^2$ metric on $\mathcal L $ as in Theorem
		\ref{teoL2}. Hence the metric $h$ and its Chern connection, can also
		be gotten from a Quillen metric on the universal curve, modified by
		Faltings $\delta$ invariant. In other words, we have
		\begin{gather*}
			\frac{\sqrt{-1}}{2} j^* \om_S\,=\, \erre (\mathcal L,\, h)\, =\,
			\erre (e^{-\delta } h_Q(\Cg,\, \om_{\rm Ar} ,\, \mathcal{K},\,
			h_{\rm
				Ar}) \\
			=\, \de \debar \delta + \erre h_Q(\Cg,\, \om_{\rm Ar} ,\,
			\mathcal{K},\, h_{\rm Ar}).
		\end{gather*}
	\end{say}

\section{Further perspectives}\label{sec:further-perspectives}

Consider the section $\widehat{\eta}$ in \eqref{we}. Restricting it $\Delta_3\,=\, 3\Delta\, \subset\,
C\times C$ we had the projective structure $\beta^\eta(C)$ on $C$ (see \eqref{be}). A projective structure
on $C$ is same as a $\text{PSL}(2, {\mathbb C})$--oper on $C$; see \cite{BD1}, \cite{BD2} for opers.
For $n\, \geq\, 2$, restricting $\widehat{\eta}$ to $$\Delta_{n+1}\,=\, (n+1)\Delta\, \subset\,
C\times C$$ we get a canonical $\text{PSL}(n, {\mathbb C})$--oper on $C$; to clarify, this 
$\text{PSL}(n, {\mathbb C})$--oper on $C$ does not depend on anything. So over $\Mg$ we get a
$C^\infty$ section of the family of $\text{PSL}(n, {\mathbb C})$--opers. We hope to be able to study
in future this section of the family of $\text{PSL}(n, {\mathbb C})$--opers over $\Mg$.

In view of the correspondence in the fundamental diagram
\eqref{intro-diag} it is reasonable to ask several other questions. A question would
be to try to understand the positions of other families of
projective structures, metrics or connections in the diagram. For example, a new
canonical projective structure has been constructed in \cite{BGV} and
it has been shown there that it is different from the $\beta^u$ and
$\beta^\eta$. It would be interesting to frame this new canonical
projective structure in the diagram, answering such questions as: What
is its $\debar$? What is the corresponding connection?  More generally
one might use the correspondence not only to associate to canonical
projective structures the corresponding $\debar$, but also conversely
 to understand $(1,\,1)$-forms, in particular K\"ahler metrics on $\Mg$, via the
associated canonical projective structures and connections.

\section*{Acknowledgements}

We thank the referee for very helpful comments to improve the exposition. The authors would like to thank 
Jos\'e Ignacio Burgos, Robin de Jong, Paola Frediani, Gian Pietro Pirola and Ken-ichi Yoshikawa for useful 
conversations/emails on the topics of this paper.

This project was initiated during the program Topics in 
Hodge Theory (code:ICTS/hodge\_theory2023/02) at International Centre for Theoretical Sciences (ICTS), 
Bangalore. The authors wish to thank ICTS for partial support and for providing a great atmosphere for 
research.
The first author is partially supported by a J. C. Bose Fellowship (JBR/2023/000003).
The second author was partially supported by INdAM-GNSAGA, by MIUR PRIN 2022: 20228JRCYB, ``Moduli 
spaces and special varieties'' and by FAR 2016 (Pavia) ``Variet\`a algebriche, calcolo algebrico, grafi 
orientati e topologici''.  The third author was partially supported by the Dutch Research Council NWO 
project Cohomology of Moduli Space of Curves (BM.000230.1) and by INdAM-GNSAGA project CUP E55F22000270001.

\section*{Mandatory declarations}

No data were generated or used.

\end{document}